\documentclass[12pt, twoside]{article}
\usepackage{amsmath,amsthm,amssymb}
\usepackage{times}
\usepackage{enumerate}
\usepackage{bm}
\usepackage[all]{xy}
\usepackage{mathrsfs}
\usepackage{amscd}
\usepackage{color}
\def\titlerunning#1{\gdef\titrun{#1}}
\makeatletter
\def\author#1{\gdef\autrun{\def\and{\unskip, }#1}\gdef\@author{#1}}
\def\address#1{{\def\and{\\\hspace*{18pt}}\renewcommand{\thefootnote}{}%
		\footnote {#1}}%
	\markboth{\autrun}{\titrun}}
\makeatother
\def\email#1{e-mail: #1}

\def\keywords#1{\par\medskip
	\noindent\textbf{Keywords.} #1}

\newtheorem{theorem}{Theorem}[section]
\newtheorem{corollary}[theorem]{Corollary}
\newtheorem{lemma}[theorem]{Lemma}
\newtheorem{proposition}[theorem]{Proposition}
\theoremstyle{definition}
\newtheorem{definition}[theorem]{Definition}
\newtheorem{remark}[theorem]{Remark}

\theoremstyle{example}

\numberwithin{equation}{section}

\xyoption{all}

\def \N {\mathbb{N}}

\def \Z {\mathbb{Z}}

\def \Z {\mathcal{Z}}

\def \a {\alpha }
\def \b {\beta}

\def \de {\delta}
\def \De {\Delta}

\def \La {\Lambda}
\def\w {\omega}
\def\Om{\Omega}

\def\na {\nabla}
\def\Ga{\Gamma}

\begin{document}
\baselineskip=17pt
	
\titlerunning{$L^{2}$-hard Lefschetz complete symplectic manifolds}
\title{$L^{2}$-hard Lefschetz complete symplectic  manifolds}
	
\author{Teng Huang and Qiang Tan}
	
\date{}
	
\maketitle
	
\address{T. Huang: School of Mathematical Sciences, University of Science and Technology of China; CAS Key Laboratory of Wu Wen-Tsun Mathematics,  University of Science and Technology of China, Hefei, Anhui, 230026, People’s Republic of China; \email{htmath@ustc.edu.cn;htustc@gmail.com}}
\address{Q. Tan: Faculty of Science, Jiangsu University, Zhenjiang, Jiangsu 212013, People’s Republic of China; \email{tanqiang@ujs.edu.cn}}

\begin{abstract}
For a complete symplectic manifold $M^{2n}$, we define the $L^{2}$-hard Lefschetz property on $M^{2n}$. We also prove that the complete symplectic manifold $M^{2n}$ satisfies $L^{2}$-hard Lefschetz property if and only if every class of $L^{2}$-harmonic forms contains a $L^{2}$ symplectic harmonic form. As an application, we get if $M^{2n}$ is a closed symplectic parabolic manifold which satisfies the hard Lefschetz property, then its Euler characteristic satisfies the inequality $(-1)^{n}\chi(M^{2n})\geq0$.
\end{abstract} 

\section{Introduction}
Let $(M,\w)$ be a $2n$-dimensional symplectic manifold. A smooth form  $\a$ on $M$ is called symplectically harmonic if $d\a=d^{\La}a=0$. Symplectic Hodge theory was introduced introduced by Brylinski \cite{Brylinski}. Further he conjectured that on a compact symplectic manifold, every de Rham cohomology class contains a harmonic representative. Some evidence for his conjecture was presented in his paper \cite{Brylinski} and
he proved the conjecture for closed K\"{a}hler manifolds. Brylinski's conjecture is equivalent to the question of the existence of a Hodge decomposition in the symplectic sense. However, a symplectic version of the above result does not hold in general. Mathieu gave two ways to give counter examples to Brylinski’s conjecture. Mathieu \cite{Mathieu} proved that any de Rham cohomology class has a (not necessarily unique) symplectically harmonic representative if and only if $(M, \w)$ satisfies the hard Lefschetz property, i.e., the map 
\begin{equation}\label{E3}
L^{n-k}: H^{k}_{dR}(M)\rightarrow H^{2n-k}_{dR}(M)
\end{equation}given by $L^{n-k}[\a]= [\a\wedge\w^{n-k}]$ is onto for all $k\leq n-1$. His proof involves the representation theory of quivers and Lie superalgebras. In fact, Mathieu's theorem is a generalization of the Hard Lefschetz Theorem for compact K\"{a}hler manifolds. Yan in \cite{Yan} given a simpler, more direct, proof of this fact, it follows the idea of the standard proof of the Hard Lefschetz Theorem (see \cite{GH}). In \cite{FMU}, the authors dealt with the symplectic manifolds satisfying a weak property following \cite{FM}, they said that $(M,\w)$ is an $s$-Lefschetz symplectic manifold, $0<s<n-1$, if (\ref{E3}) is an epimorphism for all $k<s$. In \cite{TW}, they discussed the hard Lefschetz condition on various cohomology groups and verify them for the Nakamura manifold of completely solvable type and the Kodaira-Thurston manifold.  In \cite{HT}, the author studied the $L^{2}$ cohomology of complete almost K\"{a}hler manifold. They proved that the reduced $L^{2}$ cohomology group of degree $2$ decomposes as direct sum of the closure of the invariant and anti-invariant $L^{2}$-cohomology. 

We denote by $$H^{k}_{(2);hr}(M,\w)=\{\a\in\Om^{k}_{(2)}(M): d\a=d^{\La}\a=0\}$$ the space of the $L^{2}$ symplectic harmonic $k$-forms. Let us denote by $\tilde{H}^{k}_{(2),hr}(M,\w)$ the space of $L^{2}$-harmonic cohomology in degree $k$, that is,the subspace of the $L^{2}$-harmonic forms  $\mathcal{H}_{(2)}^{k}(M)\simeq H^{k}_{(2)}(M)$ consisting of all classes which contain at least one $L^{2}$ symplectic harmonic $k$-form. Here $H^{k}_{(2)}(M)$ is the $L^{2}$-coholomogy group of degree $k$. We call that a symplectic manifold $(M,\w)$ of dimension $2n$ is said to  satisfy the $L^{2}$-hard Lefschetz property, if the map 
$$L^{n-k}: H^{k}_{(2)}(M)\rightarrow H^{2n-k}_{(2)}(M)$$
is an onto for all $k\leq n-1$. 

Let $(M,\w)$ be a closed symplectic manifold. Choose a $\w$-compatible almost complex structure $J$ on $M$. Define an almost K\"{a}hler metric, $g(\cdot,\cdot)=\w(\cdot,J\cdot)$, on $M$. Then the triple $(g, J,\w)$ is called an almost K\"{a}hler structure on $M$ and the quadruple $(M, g, J,\w)$ is called a closed almost K\"{a}hler manifold. Following the idea in \cite{Yan}, we prove the following result.
\begin{theorem}\label{T4}
Let $(M,g,J,\w)$ be a $2n$-dimensional complete non-compact almost K\"{a}hler manifold. Then the following statements are equivalent:\\
(1) Any class of $L^{2}$-harmonic forms contains a $L^{2}$ symplectic -harmonic form,\\
(2)  For any $k\leq n$, the cup product $L^{k}:H^{n-k}_{(2)}(M)\rightarrow H^{n+k}_{(2)}(M)$ is surjective.
\end{theorem}
A closed symplectic manifold $(M,\w)$ is said to satisfy the $dd^{\La}$-Lemma if every $d$-exact, $d^{\La}$-closed form is $dd^{\La}$-exact. In fact, it turns out that  $(M,\w)$ satisfies  the $dd^{\La}$-lemma if and only if the hard Lefschetz condition holds on $(M,\w)$. In \cite{TWZ}, the authors defined the $L^{2}$-$dd^{\La}$ lemma on complete symplectic manifold. They also proved that the $L^{2}$-$dd^{\La}$ lemma holds on a universal covering space of a close symplectic manifold which satisfies the hard Lefschetz property.  Combining with theorem \ref{T4}, we have a result as follows.
\begin{theorem}\label{T6}
Let $(M,g,J,\w)$ be a $2n$-dimensional closed almost K\"{a}hler manifold. We denote by $\pi: (\tilde{M},\tilde{g},\tilde{J},\tilde{\w}) \rightarrow (M, g, J,\w)$ the universal covering map. If $(M, g, J,\w)$ satisfies the hard Lefschetz property, then $(\tilde{M},\tilde{g},\tilde{J},\tilde{\w})$ satisfies the $L^{2}$-hard Lefschetz property. 
\end{theorem}
\begin{remark}
The covering space $(\tilde{M},\tilde{g},\tilde{J},\tilde{\w})$ satisfies the $L^{2}$-hard Lefschetz property if only if  $L^{n-k}:H^{k}_{(2)}(M)\rightarrow H^{2n-k}_{(2)}(M)$ is surjective for all $k\leq n$. In a complete non-compact case we actually have that $L^{n-k}$ are isomorphisms because of Poincar\'{e} duality.
\end{remark}
The last results of this article related to a well-know problem, attributed to Hopf, to the effect that the Euler number $\chi(M^{2n})$ of a compact Riemannian manifold $M^{2n}$ of negative section curvature must satisfy the inequality $(-1)^{n}\chi(M^{2n})>0$. This conjecture is true in dimensions $2$ and $4$ \cite{Chern} and it has been verified in the K\"{a}hler case for all $n$ by Gromov \cite{Gromov}. Dodziuk \cite{Dodziuk} and Singer have proposed to settle the Hopf conjecture using the Atiyah index theorem for coverings (see \cite{Atiyah}). In this approach, one is required to prove a vanishing theorem for $L^{2}$ harmonic $k$-forms, $k\neq n$, on the universal covering of $M^{2n}$. The vanishing of these  $L^{2}$ Betti numbers implies that  $(-1)^{n}\chi(M^{2n})\geq0$. The strict inequality $(-1)^{n}\chi(M^{2n})>0$  follows provided one can establish the existence of nontrivial $L^{2}$ harmonic $n$-forms on the universal cover. 

The program outlined above was carried out by Gromov \cite{Gromov} when the manifold in question is K\"{a}hler and is homotopy equivalent to a compact  manifold with strictly negative sectional curvatures.  Gromov \cite{Gromov} points out that if the Riemannian manifold $(M,g)$ is a complete simply-connected manifold and it has strictly negative sectional curvatures, then every smooth bounded closed form of degree $k\geq2$ is $d$(bounded). Therefore, Gromov introduced the notion of K\"{a}hler hyperbolicity, i.e., the K\"{a}hler metric whose K\"{a}hler form $\w$ is $\tilde{d}$(bounded). Then he proved the Hopf conjecture in the K\"{a}hler case.  In order to attack Hopf Conjecture in the K\"{a}hlerian case when  $K\leq0$ by extending Gromov’s idea, Cao-Xavier \cite{CX} and Jost-Zuo  \cite{JZ} independently introduced the concept of K\"{a}hler non-ellipticity,  which includes nonpositively curved compact K\"{a}hler manifolds, and showed  that their Euler characteristics have the desired property. In \cite{Huang}, the author proved the Hopf conjecture in some locally conformally K\"{a}hler manifolds case.

For symplectic case, inspired by K\"{a}hler geometry, Tan-Wang-Zhou \cite{TWZ} gave the definition of symplectic parabolic manifold. 
\begin{definition}
A closed almost K\"{a}hler manifold $(M,g,J,\w)$ is called symplectic parabolic if the lift $\tilde{\w}$ of $\w$ to the universal covering $(\tilde{M},\tilde{g},\tilde{J},\tilde{\w})\rightarrow(M, g, J,\w)$ is $d$(sublinear) on $(\tilde{M},\tilde{g},\tilde{J},\tilde{\w})$.
\end{definition}
The authors want to consider Hopf conjecture on a closed symplectic parabolic manifold \cite{TWZ}. One of the powerful tools for Gromov achieving Hopf conjecture on a K\"{a}hler manifold is that the Lefschetz operator $L$ commutes with the Hodge Laplacian operator $\De =dd^{\ast}+d^{\ast}d$. But in general, the Lefschetz operator $L$ does not commute with $\De$ on symplectic manifold. By considering Tseng-Yau’s new symplectic cohomologies on symplectic parabolic manifold, they got some interesting results \cite{TY}. At last, with the hard Lefschetz property which ensures that de Rham cohomology consists with the new symplectic cohomology, they obtained the result as follows.
\begin{theorem}(\cite[Theorem 1.5]{TWZ})\label{T2}
If $(M,\w)$ is a $2n$-dimensional closed symplectic parabolic manifold which 
satisfies the hard Lefschetz property, then the Euler number satisfies 
$(-1)^{n}\chi(M^{2n})\geq0$.
\end{theorem}
In this article, we use an other way to prove Theorem \ref{T2}. By the Lemma 3 in \cite{CX},  we can also get: Let $M$ be a closed $2n$-Riemannian manifold of non-positive sectional curvature. If $M^{2n}$ is homotopy equivalent with a closed symplectic manifold which satisfies the hard Lefschetz property, then the Euler number of $M$ satisfies the inequality $(-1)^{n}\chi(M^{2n})\geq 0$, see Theorem \ref{T7}.

\section{$L^{2}$-cohomology}
We recall some basic on $L^{2}$ harmonic forms \cite{Car1,Car2}.  Let $M$ be a smooth manifold of dimension $n$, let $\Om^{k}(M)$ and $\Om^{k}_{0}(M)$ denote the smooth $k$-forms on $M$ and the smooth $k$-forms  with compact support on $M$, respectively. We assume now that $M$ is endowed with a Riemannian metric $g$. Let $\langle,\rangle$ denote the pointwise inner product on $\Om^{k}(M)$ given by $g$.  The global inner product is  defined
$$(\a,\b)=\int_{M}\langle\a,\b\rangle d{\rm{Vol}}_{g}.$$
We also write $|\a|^{2}=\langle \a,\a\rangle$, $\|\a\|^{2}=\int_{M}|\a|^{2}dVol_{g}$, and let $$\Om^{k}_{(2)}(M)=\{\a\in\Om^{k}(M):\|\a\|^{2}<\infty \}.$$
The operator of exterior differentiation is $d:\Om^{k}_{0}(M)\rightarrow\Om^{k+1}_{0}(M)$, and it satisfies $d^{2}=0$; its formal adjoint is $d^{\ast}:\Om^{k+1}_{0}(M)\rightarrow\Om^{k}_{0}(M)$; we have
$$\forall\a\in\Om^{k}_{0}(M),\ \forall\b\in\Om^{k+1}_{0}(M),\ \int_{M}\langle d\a,\b\rangle=\int_{M}\langle\a,d^{\ast}\b\rangle.$$
The space $Z_{2}^{k}(M)$ is defined as follows \cite[Section 2]{Car2}:
$Z_{2}^{k}(M)$ is the kernel of the operator $d$ acting on $\Om^{k}_{(2)}(M)$. That is to say
$$Z_{2}^{k}(M)=\{\a\in\Om^{k}_{(2)}(M): d\a=0 \},$$
where  the equation $d\a=0$ has to be understood in the distribution sense, i.e., $\a\in\Z^{k}_{2}(M)$ if only if 
$$\forall\b\in\Om^{k}_{0}(M),\ (\a,d^{\ast}\b)=0.$$
Hence we have 
$$Z_{2}^{k}(M)=\big{(}d^{\ast}(\Om^{k+1}_{0}(M))\big{)}^{\bot}.$$
We can also define
\begin{equation*}
\begin{split}
\mathcal{H}_{(2)}^{k}(M)&= (d^{\ast}(\Om_{0}^{k+1}(M))^{\bot}\cap(d(\Om^{k-1}_{0}(M)))^{\bot}\\
&=Z_{2}^{k}(M)\cap\{\a\in \Om^{k}_{(2)}(M): d^{\ast}\a=0 \}\\
&=\{\a\in \Om^{k}_{(2)}(M): d\a=d^{\ast}\a=0 \}.\\
\end{split}
\end{equation*}
Because the operator $d+d^{\ast}$ is elliptic, we have by elliptic regularity: $\mathcal{H}^{k}_{(2)}(M)\subset\Om^{k}(M)$. We also remark that by the definition we have
$$\forall \a\in\Om^{k-1}_{0}(M),\ \forall\b\in\Om^{k+1}_{0}(M),\ \int_{M}\langle d\a,d^{\ast}\b\rangle=\int_{M}\langle dd\a,\b\rangle=0.$$
Hence $$(d(\Om^{k-1}_{0}(M)))^{\bot}\perp(d^{\ast}(\Om_{0}^{k+1}(M))^{\bot}$$
and we get the Hodge-de Rham-Kodaira orthogonal decomposition of $\Om^{k}_{(2)}(M)$
$$\Om^{k}_{(2)}(M)=\mathcal{H}^{k}_{(2)}(M)\oplus\overline{d(\Om^{k-1}_{0}(M))}\oplus\overline{d^{\ast}(\Om^{k+1}_{0}(M))},$$
where the closure is taken with respect to the $L^{2}$ topology.

We also define the domain of $d$ by
$$\mathcal{D}^{k}(d)=\{\a\in\Om^{k}_{(2)}(M): d\a\in L^{2}\}$$
that is to say $\a\in\mathcal{D}^{k}(d)$ if only if there is a positive constant $C$ such that
$$\forall\b\in\Om^{k}_{0}(M), |(\a,d^{\ast}\b)|\leq C\|\b\|.$$
We remark that we always have $d\mathcal{D}^{k-1}(d)\subset Z^{k}_{2}(M)$. We define the $k$-space of reduced $L^{2}$-cohomology by 
$$H_{(2)}^{k}(M)=Z_{2}^{k}/ \overline{d\mathcal{D}^{k-1}(d)}.$$
The $k$-space of non reduced $L^{2}$-cohomlogy is defined by
$$^{nr}H_{(2)}^{k}(M)=Z_{2}^{k}/ {d\mathcal{D}^{k-1}(d)}.$$
We recall a result in  \cite[Lemma 1.5]{Car1} as follows: 
$$\overline{d\mathcal{D}^{k}(d)}=\overline{\Om^{k}_{0}(M)}.$$
Therefore, the space of harmonic $L^{2}$-forms computes the reduced $L^{2}$-cohomology:
$$H^{k}_{(2)}(M)\simeq\mathcal{H}^{k}_{(2)}(M).$$
From now on, we will refer to reduced $L^{2}$-cohomology as $L^{2}$-cohomology, except where it may cause ambiguity.
\section{$L^{2}$-hard Lefschetz property}
Through out this section, $(M,\w)$ is a complete symplectic manifold. We denote by $\Om^{k}(M)$ the space of smooth $k$-forms on $M$. Let $$L:\Om^{k}(M)\rightarrow\Om^{k+2}(M)$$
denote the Lefschetz operator and  
$$\La:\Om^{k}(M)\rightarrow \Om^{k-2}(M)$$
denote the dual Lefschetz operator. We also define a second differential operator $$d^{\La}:=d\La-\La d.$$ 
Noting that $d^{\La}:\Om^{k}(M)\rightarrow\Om^{k-1}(M)$, decreasing the degree of forms by $1$ \cite{TY}.

We denote by $\mathfrak{X}(M)$  the space of vector filed on $M$. If the almost complex structure $J$ on $M$ satisfies the conditions $\w(X,JX)>0$, $\forall X\in\mathfrak{X}(M)-\{0\}$ and $\w(JX,JY)=\w(X,Y)$, $\forall X,Y\in\mathfrak{X}(M)$, then the almost complex structure $J$ is said to be compatible with the symplectic form $\w$. There is a well-defined Hermitian metric $g(X,Y)=\w(X,JY)$ given by above two conditions. We can use the Hermitian metric $g$ to define a Hodge star operator $\ast$. Therefore, the dual Lefschetz operator $\La$ is then just the adjoint of $L$, i.e, $\La=L^{\ast}=(-1)^{k}\ast L\ast$ \cite{TY}.

A form $\a$ is called symplectic harmonic if it satisfies $d\a=d^{\La}a=0$. Brylinski conjectured that on a closed symplectic manifold, every de Rham cohomology class contains a symplectic harmonic representative. Several years later, his conjecture for closed symplectic manifolds was disproved by Oliver Mathieu \cite{Mathieu}. In fact, Mathieu proved that every de Rham cohomology
$H^{\ast}_{dR}(M)$ class contains a symplectic harmonic form if and only if the symplectic manifold satisfies the hard Lefschetz property.
\begin{proposition}(\cite[Proposition 7]{Mathieu} and \cite{Yan})
Let $(M^{2n},\w)$ be a $2n$-dimensional symplectic manifold (not necessarily compact). Then the following two assertions are equivalent:\\
(1) Any cohomology class contains a harmonic cocycle;\\
(2) For any $k\leq n$, the cup product $L^{k}:H^{n-k}_{dR}(M;\mathbb{R})\rightarrow H^{n+k}_{dR}(M;\mathbb{R})$ is surjective.
\end{proposition}
We denote by $H^{k}_{(2);hr}(M,\w)=\{\a\in\Om^{k}_{(2)}(M): d\a=d^{\La}\a=0\}$ the space of the  $L^{2}$ symplectic harmonic $k$-forms.  We call that a symplectic manifold $(M,\w)$ of dimension $2n$ is said to be $L^{2}$-hard Lefschetz property, if the map $L^{n-k}: H^{k}_{(2)}(M)\rightarrow H^{2n-k}_{(2)}(M)$ is an onto for all $k\leq (n-1)$. 
\begin{definition}
A differential $k$-form $B_{k}$ with $k\leq n$ is called primitive, i.e., $B_{k}\in P^{k}(M)$, if it satisfies the two equivalent conditions: (i) $\La B_{k}=0$; (ii) $L^{n-k+1}B_{k}=0$.
\end{definition} 
Noting that $d^{\La}B_{k}=[d,\La]B_{k}=-\La dB_{k}$ for any $B_{k}\in P^{k}(M)$, we then have
\begin{lemma}\label{L2}
Any closed $L^{2}$-primitive form is $L^{2}$ symplectic harmonic. 
\end{lemma}
Given any $k$-form, there is a unique Lefschetz decomposition into primitive forms \cite{TY}. Explicitly, we shall write
\begin{equation}\label{E2}
 A_{k}=\sum_{r\geq\max(k-n,0)}\frac{1}{r!}L^{r}B_{k-2r},
\end{equation}
where each $B_{k-2r}$ can be written in terms of $A_{k}$ as
\begin{equation}\label{E4}
B_{k-2r}=\Phi_{(k,k-2r)}(L,\La)A_{k}\equiv(\sum_{s=0}a_{r,s}\frac{1}{s}!L^{s}\La^{r+s})A_{k},
\end{equation}
where the operator $\Phi_{(k,k-2r)}(L,\La)$ is a linear combination of $L$ and $\La$ with the rational coefficients $a_{r,s}$'s dependent only on $(d, k, r)$. Consider now the action of the differential operators on a $k$-form written in Lefschetz decomposed form of \ref{E2}, see \cite{TY} Section 2.2. We have
\begin{equation}\label{E5}
dA_{k}=\sum \frac{1}{r!}dB_{k-2r},\ d^{\La}A_{k}=\sum \frac{1}{r!}L^{r}(dB_{k-2r-2}+d^{\La}B_{k-2r}),
\end{equation} 
where $B_{k}\in P^{k}(M)$. The first is  simply due to the fact that $d$ commute with $L$. The second follows from commuting $d^{\La}$ through $L^{r}$ and repeatedly applying the relation that $[d^{\La},L]=d$.
\begin{proposition}(\cite[Proposition 2.6]{TY})\label{P2}
Let $A_{k}\in\Om^{k}(M)$ and $B_{k-2r}\in P^{\ast}(M)$ be its Lefschetz decomposed primitive forms. Then $dA_{k}=d^{\La}A_{k}=0$ if and only if $dB_{k-2r}=0$, for all $r$.	
\end{proposition}
\begin{proof}
Starting with (\ref{E4}), we apply the exterior derivative $d$ to it. Commuting $d$ through $\Phi_{(k,k-2r)}(L,\La)$, the $d$- and $d^{\La}$-closedness of $A_{k}$ immediately implies $dB_{k-2r}=0$. Assume now $dB_{k-2r}=0$, for all $B_{k-2r}$. Note that this trivially also implies $d^{\La}B_{k-2r}=-\La dB_{k-2r}=0$. With the Equations in (\ref{E5}), we therefore find $dA_{k}=d^{\La}A_{k}=0$.	
\end{proof}
By the Proposition \ref{P2}, we prove the duality on $L^{2}$ symplectic harmonic forms.
\begin{lemma}\label{L3}
Let $(M,g,J,\w)$ be a $2n$-dimensional complete non-compact almost K\"{a}hler manifold. Then the map 
$$L^{n-k}: H^{k}_{(2),hr}(M,\w)\rightarrow H^{2n-k}_{(2),hr}(M,\w)$$
is an isomorphism for $k\geq0$.
\end{lemma}
\begin{proof}
Since $\w$ is bounded in $M$, the map $L^{n-k}:\Om^{k}_{(2)}(M)\rightarrow\Om^{2n-k}_{(2)}(M)$ is well define.  First, by the Jocobi identity, we observe that 
$$[d^{\La},L^{i}]=[[d,\La],L^{i}]=[[\La,L^{i}],d]=i(n+1-k-i)[L^{i-1},d],$$
here we use the identity (see \cite{Huybrechts} Corollary 1.2.28), for any $k$-form on $M$,
$$[L^{i},\La](\cdot)=i(k-n+i-1)L^{i-1}(\cdot).$$
Therefore, we have
$$d(L^{k}\a)=0,\ d^{\La}(L^{k}\a)=0,\ \forall\a\in H^{k}_{(2),hr}(M,\w).$$
Suppose that $A_{2n-k}=\sum_{r\geq n-k}\frac{1}{r!}L^{r}B_{k-2r}$ satisfies $dA_{2n-k}=d^{\La}A_{2n-k}=0$, following Proposition 2.6, it implies that $dB_{k-2r}=0$. Note that $d^{\La}B_{k-2r}=-\La dB_{k-2r}=0$. Therefore, we have $dL^{i}B_{k-2r}=0$ and $d^{\La}L^{i}B_{k-2r}=0$. Hence $A_{k}:=\sum_{r\geq n-k}\frac{1}{r!}L^{r-n+k}B_{k-2r}$ is in $H^{k}_{(2),hr}(M,\w)$ and $A_{2n-k}=L^{n-k}A_{k}$. This proves that the homomorphism is surjective. The homomorphism is also injective, since the map $L^{n-k}$ is injective.  
\end{proof} 
Before going on to the study of the $L^{2}$-hard Lefschetz property, we need
to recall the splitting of the cohomology groups in terms of the primitive classes proved by Yan \cite{Yan} for hard Lefschetz symplectic manifolds.
\begin{lemma}\label{L4}
Let $(M,g,J,\w)$ be a $2n$-dimensional complete non-compact almost K\"{a}hler manifold. If $M$ has $L^{2}$-hard Lefschetz property, then there is a splitting
$$H^{k}_{(2)}=P^{k}_{(2)}\oplus L(H^{k-2}_{(2)}(M))$$
where $P^{k}_{(2)}(M)$ is given by $P^{k}_{(2)}(M)=\{v\in H^{k}_{(2)}(M): L^{n-k+1}v=0\}$, for all $k\leq n$.
\end{lemma}
\begin{proof}
First, let us see that $P^{k}_{(2)}(M)\cap \textit{Im}L=0$. Take $\a\in P^{k}_{(2)}(M)$ with $\a=L\b$, $\b\in H^{k-2}_{(2)}(M)$. Then $L^{n-k+2}\b=L^{n-k+1}\a=0$. Since the map $L^{n-k+1}:\Om^{k-2}\rightarrow\Om^{2n-k+2}$ is bijective for $k-2\leq n$, $\b=0$ and hence $\a=0$.

Now let us consider $\a\in H^{k}_{(2)}(M)$ with $k\leq n$, and take the element $L^{n-k+1}\a\in H^{2n-k+2}_{(2)}(M)$. If $L^{n-k+1}\a$ is the zero class, then $\a\in P_{(2)}^{k}(M)$, and the lemma is proved. If  $L^{n-k+1}\a$ is non-zero, then there exists $\b\in H^{k-2}_{(2)}(M)$ such that $L^{n-k+1}\a=L^{n-k+2}\b$ since $(M,\w)$ is $L^{2}$-hard Lefschetz and so map $L^{n-k+2}: H^{k-2}_{(2)}(M)\rightarrow H^{2n-k+2}_{(2)}$ is an isomorphism.  Hence $\a-L\b\in P^{k}_{(2)}(M)$. But $\a=(\a-L\b)+L\b$, which lies in $P^{k}_{(2)}(M)\oplus\textit{Im}L$.
\end{proof}
\begin{proof}[\textbf{Proof of Theorem \ref{T4}}.]
$(1)\Rightarrow(2)$. Consider the following commutative diagram,\\
\centerline{
\xymatrix{
H_{(2),hr}^{n-k}(M)\ar[r]^-{L^{k}}\ar[d]& H^{n+k}_{(2),hr}(M)\ar[d] \\
H^{n-k}_{(2)}(M) \ar[r]^-{L^{k} } & H^{n+k}_{(2)}(M) 
		}
}
where the two vertical arrows are surjective. It follows from Lemma \ref{L3} that the second horizontal arrow is also surjective.
	
$(2)\Rightarrow(1)$. Now we assume that $L^{k}: H^{n-k}_{(2)}(M)\rightarrow H^{n+k}_{(2)}(M)$ is surjective for all $k\leq n$. Following Lemma \ref{L4}, it implies that at $H^{n-k}_{(2)}(M)=\textit{Im}L+P^{n-k}_{(2)}$.

The rest of the proof is an induction on the degree of cohomology classes on $M$. It is easy to see that any $L^{2}$-harmonic form $\a$ of degree $\leq 1$ must be $L^{2}$ symplectic harmonic form, since $$d^{\La}\a=-\La(d\a)=0.$$
We suppose that when $r<n-k$, any class $\b\in H^{r}_{(2)}(M)$ contains a $L^{2}$ symplectic harmonic form. We need to show that any class in $H^{n-k}_{(2)}(M)$ also contains a symplectic $L^{2}$-harmonic form. By induction, we already know that any class in $\textit{Im}L$ contains a $L^{2}$ symplectic harmonic form. So it sufficies to show that any cohomology class $v\in P_{n-k}$ contains a symplectic $L^{2}$-harmonic.
	
Let $v=[z]$, $z\in\Om^{n-k}_{(2)}(M)$ is a closed form. Since $[v\wedge\w^{k+1}]=0$ in $H_{(2)}^{n+k+2}(M)$, there is a sequence $(n+k+1)$-forms $\{\gamma_{i}\}\in\Om^{n+k+1}_{(2)}(M)$ such that $$z\wedge\w^{k+1}=\lim d\gamma_{i}.$$ 
Since $L^{k+1}:\Om^{n-k-1}_{(2)}(M)\rightarrow\Om^{n+k+1}_{(2)}(M)$ is onto, we can pick $\theta_{i}\in\Om^{n-k-1}_{(2)}(M)$ such that $$\gamma_{i}=\theta_{i}\wedge\w^{k+1}.$$ 
Then $(z-\lim d\theta_{i})\wedge\w^{k+1}=0$. Thus, if we write $u$ for $z-\lim d\theta_{i}$, then $[u]=[z]=v$ and $L^{k+1}u=0$. Hence $u$ is primitive and closed. According to Lemma \ref{L2}, $u$ is also symplitically harmonic. 
\end{proof}
Following the proof of Theorem \ref{T4}, we have
\begin{corollary}\label{C1}
Let $(M,g,J,\w)$ be a $2n$-dimensional complete non-compact almost K\"{a}hler manifold and $k=0,1,2$. Then any classes in $H^{k}_{(2)}(M)$ contains a $L^{2}$ symplectic harmonic form.
\end{corollary}
A closed symplectic manifold $(M,\w)$ is said to satisfy the $dd^{\La}$-Lemma if every $d$-exact, $d^{\La}$-closed form is $dd^{\La}$-exact. In fact, it turns out that the following conditions are equivalent on a closed symplectic manifold $(M,\w)$ \cite{AT,Guillemin,Mathieu,Yan}:\\
(1) $(M,\w)$ satisfies the $dd^{\La}$-Lemma;\\
(2) every de Rham cohomology class admits a representative being both $d$-closed and $d^{\La}$-closed;\\
(3) the hard Lefschetz condition holds on $(M,\w)$.

In \cite{TWZ}, they defined the $L^{2}$-$dd^{\La}$ Lemma on a complete non-compact almost K\"{a}hler manifold.
\begin{definition}(\cite[Definition 3.3]{TWZ})
Let $(M, g, J,\w)$ be a complete non-compact almost K\"{a}hler manifold with bounded geometry. Let $\a\in\Om^{k}_{(2)}$ be a $d$- and $d^{\La}$-closed differential form. We say that the $L^{2}$-$dd^{\La}$ Lemma holds if the following properties are equivalent:\\
(i) $\a=d\b$, $\b\in L^{2}_{1}\Om^{k-1}$;\\
(ii) $\a=d^{\La}\gamma$, $\gamma\in L^{2}_{1}\Om^{k+1}$;\\
(iii) $\a=dd^{\La}\theta$, $\theta\in L^{2}_{2}\Om^{k}$.
\end{definition}
In \cite{TWZ}, they prove that the $L^{2}$-$dd^{\La}$ lemma holds on universal covering space of a closed symplectic manifold which satisfies the hard Lefschetz property.
\begin{proposition}\label{P3}(\cite[Proposition 3.4]{TWZ})
Let $(M,g,J,\w)$ be a $2n$-dimensional closed almost K\"{a}hler manifold. We denote by $\pi: (\tilde{M},\tilde{g},\tilde{J},\tilde{\w}) \rightarrow (M, g, J,\w)$ the universal covering map. If $(M, g, J,\w)$ satisfies the hard Lefschetz property, then $L^{2}$-$dd^{\La}$ Lemma holds on $(\tilde{M},\tilde{g},\tilde{J},\tilde{\w})$.
\end{proposition}
Using the useful Proposition \ref{P3}, we could prove that the any class of $L^{2}$-harmonic forms on the universal covering space $(\tilde{M},\tilde{g},\tilde{J},\tilde{\w})$ contains a $L^{2}$ symplectic harmonic form, see Proposition \ref{P4}. Before our proof, we construct some estimates on $L^{2}$-harmonic forms.
\begin{lemma}
If $\a$ is a $L^{2}$-harmonic $p$-form on a complete Riemannian manifold $(M,g)$ with bounded geometry, then $\na^{k}\a\in L^{2}$ for all $k\in\N$. 
\end{lemma}
\begin{proof}
Following the Weitzenb\"{o}ck formula for the harmonic $p$-form $\a$, we have
\begin{equation}\label{E6}
\na^{\ast}\na\a+ \{Riem,\a\}=0,
\end{equation}
Here $Riem$ denotes the Riemann curvature tensor of $(M,g)$. Assume now that the sequence of cutoff functions, $\{f_{i}\}_{i=1}^{\infty}$, is $C^{1}$ bounded for all $i$, i.e., there exists a positive constant $C$ such that $|\na f_{i}|\leq C$. Then have
\begin{equation}
\begin{split}
0&=(\na^{\ast}\na\a+ \{Riem,\a\},f_{i}^{2}\a)\\
&=\|\na (f_{i}\a)|^{2}-\|(\na f_{i})\a\|^{2}+(\{Riem,\a\},f_{i}^{2}\a).\\
\end{split}
\end{equation}
Since $Riem$ is bounded, taking $i\rightarrow\infty$, we obtain $\na\a\in L^{2}$. Differentiating equation (\ref{E6}), we obtain
\begin{equation}\label{E7}
0=\na^{\ast}\na\na\a-[\na,\na^{\ast}]\na\a+\{\na Riem,\a\}+\{Riem, \na\a\} \end{equation}
The quantity $[\na,\na^{\ast}]\na\a$ is a sum of terms involving at most one derivative of $\na\a$ and  Riemannian curvature tensor. Hence taking the $L^{2}$-inner product of (\ref{E7}) with $\na\a$ and integrating by parts, we obtain
$$\|\na^{2}\a\|^{2}\leq c_{1}\|\na\a\|^{2}+c_{2}\|\a\|^{2}.$$
Here $c_{1},c_{2}$ are two bounded constants because $\na Riem$ and $Riem$ are bounded. We now induct. Suppose that $\na^{b}\a\in L^{2}$, for $b\leq k$. Differentiating (\ref{E6}) $k$ times, gives 
$$\na^{k}(\na^{\ast}\na\a +\{Riem,\a\})=0.$$
Nothing that
\begin{equation*}
\begin{split}
\na^{k}\na^{\ast}(\na\a)&=\na^{k-1}[\na,\na^{\ast}]\na\a+\na^{k-1}\na^{\ast}\na^{2}\a\\
&=\na^{k-1}[\na,\na^{\ast}]\na\a+\na^{k-2}[\na,\na^{\ast}]\na^{2}\a+\na^{k-2}\na^{\ast}\na^{3}\a\\
&=\cdots\\
&=\sum_{i=1}^{k}\na^{k-i}[\na,\na^{\ast}]\na^{i}\a+\na^{\ast}\na^{k+1}\a.\\
\end{split}
\end{equation*}
The quantity $\na^{k-i}[\na,\na^{\ast}]\na\a=\na^{k-i}(Riem\otimes\na\a)$ is a sum of terms involving at most $\na^{j}\a$, $j=1,\cdots,k-i+1$ and $\na^{l}Riem$, $l=0,\cdots,k-i$.
Then we have
\begin{equation*}
\begin{split}
0&=(\na^{k}(\na^{\ast}\na\a+\{Riem,\a\}),f^{2}_{i}\na^{k}\a)\\
&\geq(\na^{\ast}\na^{k+1}\a, f_{i}^{2}\na^{k}\a)-\sum_{0\leq j\leq k}C_{j}\|\na^{j}\a\|^{2},\\
&\geq\|\na (f_{i}\na^{k}\a)\|^{2}-\|df_{i}\otimes\na^{k}\a\|^{2}-\sum_{0\leq j\leq k}C_{j}\|\na^{j}\a\|^{2},\\
\end{split}
\end{equation*}
for some constants $C_{j}$ determined by the sup norms of $|\na^{b}Riem|$, $b\leq k$. Taking the limit as $i\rightarrow\infty$ gives $\na^{k+1}\a\in L^{2}$.
\end{proof}
\begin{proposition}\label{P4}
Let $(M,g,J,\w)$ be a $2n$-dimensional closed almost K\"{a}hler manifold. We denote by $\pi: (\tilde{M},\tilde{g},\tilde{J},\tilde{\w}) \rightarrow (M, g, J,\w)$ the universal covering map. If $(M, g, J,\w)$ satisfies the hard Lefschetz property, then for any $k\leq n$, $H^{k}_{(2),hr}(\tilde{M})\rightarrow H^{k}_{(2)}(\tilde{M})$ is surjective.
\end{proposition}
\begin{proof}
Let $\a$ be a $L^{2}$-harmonic $k$-form. Notice that $(\tilde{M},\tilde{g},\tilde{J},\tilde{\w})$ is a complete, non-compact almost K\"{a}hler manifold with bounded geometry. Then we can obtain that $\a$ is smooth and $\|\a\|_{L^{2}_{k}(\tilde{M})}<\infty$ for all $k\in\N$.\\
(1) If $d^{\La}\a=0$, then $\a$ is a $L^{2}$ symplectic harmonic form.\\
(2)  If $d^{\La}\a\neq 0$, then $d^{\La}\a\in L^{2}$ since $|d^{\La}\a|\leq c|\na\a|$ and $\na\a\in L^{2}$. We observe that $d(d^{\La}\a)=0$ and $d^{\La}(d^{\La}\a)=0$, following Proposition \ref{P3}, we can find $\gamma\in L^{2}_{2}(\tilde{M})$ such that $d^{\La}\a=dd^{\La}\gamma=0$. Hence  $d^{\La}(\a+d\gamma)=0$ and $d(\a+d\gamma)=0$. It follows that $\a+d\gamma$ is a $L^{2}$ symplectic harmonic form. This proves that the map is  surjective.
\end{proof}
\begin{proof}[\textbf{Proof of Theorem \ref{T6}}.]
The conclusions follow from Theorem \ref{T4} and Proposition \ref{P4}. 
\end{proof}

\section{Symplectic parabolic manifolds}
\subsection{$L^{2}$-Hodge number}
We assume throughout this subsection that $(M,g,J)$ is a compact complex $n$-dimensional manifold with a Hermitian metric $g$, and $\pi:(\tilde{M},\tilde{g},\tilde{J})\rightarrow(M,g,J)$ its universal covering with $\Gamma$ as an isometric group of deck transformations. Let $\mathcal{H}^{k}_{(2)}(\tilde{M})$ be the spaces of $L^{2}$-harmonic $k$-forms on $\Om^{k}_{(2)}(\tilde{M})$, the squared integrable $k$-forms on $(\tilde{M},\tilde{g})$, and denote by $\dim_{\Gamma}\mathcal{H}^{k}_{(2)}(\tilde{M})$ the Von Neumann dimension of $\mathcal{H}^{k}_{(2)}(\tilde{M})$ with respect to $\Gamma$ \cite{Atiyah,Pansu}. Its precise definition is not important in our article but only the
following two basic facts are needed, see \cite{Gromov,Pansu}.\\
(1)  $\dim_{\Ga}\mathcal{H}_{(2)}^{k}(M)=0 \Leftrightarrow \mathcal{H}_{(2)}^{k}(M)=\{0\},$\\
(2) $\dim_{\Ga}\mathcal{H}$ is additive. Given 
$$0\rightarrow\mathcal{H}_{1}\rightarrow\mathcal{H}_{2}\rightarrow 
\mathcal{H}_{3}\rightarrow 0,$$ 
one have 
$$\dim_{\Ga}\mathcal{H}_{2}=\dim_{\Ga}\mathcal{H}_{1}+\dim_{\Ga}\mathcal{H}_{3}.$$
We denote by $h_{(2)}^{k}(\tilde{M})$ the $L^{2}$-Hodge numbers of $M$, which are defined to be $$h_{(2)}^{k}(M):=\dim_{\Gamma}\mathcal{H}_{(2)}^{k}(\tilde{M}),\ (0\leq k\leq n).$$ 
It turns out that $h^{k}_{(2)}(M)$ are independent of the Hermitian metric $g$ and depend only on $(M,J)$. By the $L^{2}$-index theorem of Atiyah \cite{Atiyah}, we have the following crucial identities between $\chi(M)$ and the $L^{2}$-Hodge numbers $h_{(2)}^{k}(M)$:
$$\chi(M)=\sum_{k=0}^{n}(-1)^{k}h_{(2)}^{k}(M).$$
\begin{remark}
	$L^{2}$-Betti number are not homotopy invariants for complete non-compact manifolds. But Dodziuk \cite{Dodziuk1977} proves that the class of the representation of $\Gamma$ on the space of $L^{2}$-harmonic forms is a homotopy invariant of $M$. In particular the $\Gamma$-dimension (in the sense of Von Neumann) of the space of $L^{2}$-harmonic forms does not depend on the chosen $\Gamma$-invariant metric.
\end{remark}
\subsection{$d$(sublinear) form}
Let $(M,g)$ be a Riemannian manifold. A differential form $\a$ is called $d$(bounded) if there exists a form $\b$ on $M$ such that $\a=d\b$ and 
$$\|\b\|_{L^{\infty}(M,g)}=\sup_{x\in M}|\b(x)|_{g}<\infty.$$
It is obvious that if $M$ is compact, then every exact form is $d$(bounded). However, when $M$ is not compact, there exist smooth differential forms which are exact but not $d$(bounded). For instance, on $\mathbb{R}^{n}$, $\a=dx^{1}\wedge\cdots\wedge dx^{n}$ is exact, but it is not $d$(bounded). Let’s recall some concepts introduced in \cite{CX,JZ}.
\begin{definition}
	A differential form $\a$ on a complete non-compact Riemannian manifold $(M,g)$ is called $d$(sublinear) if there exist a differential form $\b$ and a number $c>0$ such that 
	$d\b=\a$
	and
	$$\ |\b(x)|_{g}\leq c(1+\rho_{g}(x,x_{0})),$$ 
	where $\rho_{g}(x,x_{0})$ stands for the Riemannian distance between $x$ and a base point $x_{0}$ with respect to $g$.
\end{definition}
We recall the following classical fact pointed proved in  \cite{CX} and for the reader’s convenience, we include a simple proof here (see also \cite{CY} Lemma 3.2):
\begin{theorem}\label{T3}
	Let $M$ be a complete simply-connected manifold of non-positive sectional curvature and $\a$ a bounded closed $k$-form on $M$, $k\geq 1$. Then $\a$ is $d$(sublinear).
\end{theorem}
\begin{proof}
	Fix $x_{0}\in M$ and denote by $\exp_{x_{0}}:T_{x_{0}}M\rightarrow M$ the 
	exponential map. Let $\phi_{t}:M\rightarrow M$, $t\in[0,1]$, be a family of 
	maps defined by $\phi_{t}(x)=\exp_{x_{0}}(t\circ\exp^{-1}_{x_{0}}(x))$, 
	$x\in M$. We denote the distance function from $x_{0}$ by $\rho$, then
	$$X_{t}|_{\phi_{t}(x)}=(\frac{d}{dt}\phi_{t})|_{\phi_{t}(x)}=\rho(x)\na\rho|_{\phi_{t}(x)}.$$
	It is clear that $\phi_{1}=id$ and $\phi_{0}\equiv x_{0}$. Then
	$$\a=\int_{0}^{1}(\frac{d}{dt}\phi^{\ast}_{t}\a)(x)dt=\int_{0}^{1}\phi^{\ast}_{t}(\mathcal{L}_{X_{t}}\a)(x)dt=d(\int_{0}^{1}\phi^{\ast}_{t}(\mathit{i}_{X_{t}}\a)),$$
	where we have use the Cartan's formula $\mathcal{L}=[d,\mathit{i}]$. If we 
	set 
	$$\b=\int_{0}^{1}\phi_{t}^{\ast}(\mathit{i}_{X_{t}}\a)dt,$$
	then $\a=d\b$. We show $\b$ has $d$(sublinear) $L^{\infty}$-norm. Fix $x\in M$, $v_{1},\cdots, v_{k-1}\in T_{x}M$, $|v_{i}|=1$, $\langle v_{i},\na\rho\rangle=0$, we have
	$$|\b(v_{1},\cdots,v_{p-1})|(x)=\int_{0}^{1}\a\big{(}X_{t},(d\phi_{t})(v_{1}),\cdots,(d\phi_{t})(v_{k-1})\big{)}(\phi_{t}(x))dt.$$
	By the standard comparison theorem , we have 
	$$|d\phi_{t}(v)|\leq ct$$
	for $v\in T_{x}M$, $|v|=1$ and $\langle v,\na\rho\rangle=0$, $c$ is a uniform positive constant. One also can see \cite{BGS,CX}. Hence
	$$|\b(v_{1},\cdots,v_{p-1})|(x)\leq c\|\a\|_{L^{\infty}(M,g)}|X_{t}|\int_{0}^{1}t^{k-1}dt\leq c\|\a\|_{L^{\infty}(M,g)}\rho(x) .$$
\end{proof}
If $(M,g)$ is a closed Riemannian manifold of non-positive curvature, then the sectional curvature on the universal covering space $(\tilde{M},\tilde{g})$ of $(M,g)$ is also non-positive. We have
\begin{proposition}\label{P1}
If $(M,g)$ is a closed smooth Riemannian manifold of non-positive sectional curvature and $\pi:(\tilde{M},\tilde{g})\rightarrow (M,g)$ its universal covering. If $\a$ is a closed $k$-form on $M$, then the lifted $k$-form $\tilde{\a}:=\pi^{\ast}\a$ is $d$(sublinear).
\end{proposition}
Gromov gave the definition of K\"{a}hler hyperbolic in \cite{Gromov}. A closed complex manifold is called K\"{a}hler hyperbolic if it admits a K\"{a}hler metric whose K\"{a}hler form $\w$ is $d$(bounded). After Gromov’s work, Cao-Xavier and Jost-Zuo gave the definition of K\"{a}hler parabolic. A closed complex manifold is called K\"{a}hler parabolic if it admits a K\"{a}hler metric whose K\"{a}hler form $\w$ is $d$(bounded). The following result is the main theorem in \cite{CX,Gromov,JZ}.
\begin{theorem}(\cite{Gromov,CX})\label{T5}
	Let $M$ be a complete $2n$-dimensional K\"{a}hler manifold with a $d$(sublinear) K\"{a}hler form $\w$. Then $\mathcal{H}^{k}_{(2)}(M,g)\neq\{0\}$,when $k\neq n$, i.e., 
	$$h_{(2)}^{k}(M)=0,\ k\neq n.$$ 
	Furthermore, if the K\"{a}hler form $\w$ is $d$(bounded), then $\mathcal{H}^{k}_{(2)}(M,g)\neq\{0\}$,when  $k=n$, i.e,
	$$h_{(2)}^{k}(M)\geq 1,\ k=n.$$
\end{theorem}
Hitchin \cite{Hitchin} has proven that the $L^{2}$ harmonic forms on a complete non-compact K\"{a}hler parabolic manifold lie in the middle dimension, that is, if the K\"{a}hler form $\w$ on a complete non-compact K\"{a}hler manifold is $d$(sublinear), then the only $L^{2}$ harmonic forms lie in the middle dimension. We first reproduce for the reader’s benefit the proof of the theorem of Jost-Zuo.
\begin{theorem}(\cite{JZ} and \cite[Theorem 1]{Hitchin})\label{T8}
Let $M$ be a complete oriented Riemannian manifold and let $\a$ be a $d$(sublinear) $p$-form i.e., there exists a $(p-1)$-form $\b$ such that
$$|\a(x)|\leq c,\ |\b(x)|\leq c(\rho(x,x_{0})+1).$$
Then for each $L^{2}$-cohomology class $[\eta]\in H^{q}_{(2)}(M)$,
$$[\a\wedge\eta]=0\in H_{(2)}^{p+q}(M).$$	
\end{theorem}
\begin{proof}
Let $B_{r}$ be the ball in $M$ with centre $x_{0}$ and radius $r$. Take a smooth function $\chi_{r}:M\rightarrow\mathbb{R}^{+}$ with $0\leq\chi_{r}(x)\leq1$,
\begin{equation*}
\chi_{r}(x)=\left\{
\begin{aligned}
1, &  & x\in B_{r} \\
0,  &  & x\in M\backslash B_{2r}
\end{aligned}
\right.
\end{equation*}
and $|d\chi_{r}(x)|\leq K/\rho$ for $x\in B_{2r}\backslash B_{r}$. Such a function may be obtained by smoothing the function $f(\rho(x,x_{0}))$  where $f(\rho)=1$ for $\rho\leq r$, $f(\rho)=2-\rho/r$ for $r\leq\rho\leq 2r$ and $f(\rho)=0$ for $\rho\geq2r$.

The form $d(\chi_{r}\b\wedge\eta)$ has compact support, so $d(\chi_{r}\b\wedge\eta)\in\Om^{p+q}_{(2)}$. We want to show that as $r\rightarrow\infty$ these forms converge in $L^{2}$ to $\a\wedge\eta$. Consider
\begin{equation}\label{E1}
d(\chi_{r}\b\wedge\eta)=d\chi_{r}\wedge\b\wedge\eta+\chi_{r}\a\wedge\eta.
\end{equation}
As $\chi_{r}=1$ on $B_{r}$, $\chi_{r}\a\wedge\eta$ converges pointwise to $\a\wedge\eta$. Moreover, $|\a(x)|\leq c$ and $\eta\in L^{2}$, so $\a\wedge\eta\in L^{2}$ and hence, since $\chi_{r}$ is bounded, $\chi_{r}\a\wedge\eta\rightarrow\a\wedge\eta$ in $L^{2}$.

Now $d\chi_{r}$ vanishes on $B_{r}$ and outside $B_{2r}$ , and on the annulus in between we have the estimates $|d\chi_{r}(x)|\leq K/\rho(x,x_{0})$ and $|\b(x)|\leq c(\rho(x,x_{0})+1)$. Then 
$$\int_{M}|d\chi_{r}\wedge\b\wedge\eta|\leq const.\int_{B_{2r}\backslash B_{r}}|\eta|^{2}\leq const.\int_{M\backslash B_{r}}|\eta|^{2}.$$
This converges to zero as $r\rightarrow\infty$ since $\eta\in L^{2}$. We thus have converges of both terms on the right-hand side of (\ref{E1}) and consequently $d(\chi_{r}\b\wedge\eta)$ converges in $L^{2}$ to $\a\wedge\eta$. Hence $\a\wedge\eta$ lies in the closure of $d\Om^{p-1}_{(2)}\cap\Om^{p}_{(2)}$ and its $L^{2}$-cohomology class vanishes.
\end{proof}

\begin{theorem}\label{T1}
Let $M^{2n}$ be a closed symplectic parabolic manifold satisfies the hard Lefschetz property. We denote by the universal covering $\pi:(\tilde{M},\tilde{g},\tilde{J},\tilde{\w})\rightarrow (M, g, J,\w)$. Then all $L^{2}$ harmonic $p$-forms on $(\tilde{M},\tilde{g},\tilde{J},\tilde{\w})$ for $p\neq n$ vanish.
\end{theorem}
\begin{proof}
Since $M$ is compact, $\w$ is bounded and $\pi$ is a local isometry, $\tilde{\w}:=\pi^{\ast}(\w)=d\eta$ is a $d$(sublinear) form on $\tilde{M}$. For any $p<n$, $\tilde{\w}^{n-p}=d(\eta\wedge\tilde{\w}^{n-p-1})$ is also a $d$(sublinear) form on $\tilde{M}$, since  $$|\tilde{\w}^{n-p}|\leq|\tilde{\w}|^{n-p}\leq c$$ and $$|(\eta\wedge\tilde{\w}^{n-p-1})(x)|\leq |\eta(x)|\cdot|\tilde{\w}|^{n-p-1}\leq c(\rho(x,x_{0})+1).$$ 
From Theorem \ref{T8}, the linear growth of $\eta\wedge\tilde{\w}^{n-p-1}$ implies that the map $L^{n-p}: H_{(2)}^{p}(M)\rightarrow H^{2n-p}_{(2)}(M)$ defined by $L([\eta])=[\w^{n-p}\wedge\eta]$ is zero.  By Theorem \ref{T6}, it means that  $H_{(2)}^{p}(M)=\{0\}$ for all $p\neq n$. Hence the only non-zero harmonic forms occur when $p=n$.
\end{proof}
\begin{remark}
By using methods of contact geometry, starting with a contact manifold $(M,\a)$ having an exact symplectic filling (see \cite{HT} Definition 3.1), the authors construct a $d$(bounded) complete almost K\"{a}hler manifold $Y$ satisfying $H^{1}_{(2)}(Y)\neq {0}$. Therefore, the condition of hard Lefshetz property could not remove under the method in the vanishing theorem \ref{T1}.
\end{remark}
\begin{corollary}
Let $M^{2n}$ be a closed symplectic manifold satisfies the hard Lefschetz property. If the sectional curvature of $M^{2n}$ is non-positive, then the Euler characteristic of $M$ satisfies the inequality $(-1)^{n}\chi(M^{2n})\geq0$.		
\end{corollary}
\begin{proof}
Let $\pi:M\rightarrow\tilde{M}$ be the universal covering map and $\w$ the symplectic form on $M$. Since $\pi^{\ast}\w$ is a bounded closed form on $M$, it follows from Proposition \ref{P1} that $\pi^{\ast}\w$ is $d$(sublinear). It follows from Theorem \ref{T1} that the $L^{2}$ cohomology of $M$ is concentrated in the middle dimension. The Atiyah index theorem for covers \cite{Atiyah} then gives $(-1)^{n}\chi(M^{2n})\geq0$.
\end{proof}
In \cite{CX}, the authors shown that the property of $d$(sublinearity) has homotopy invariance.
\begin{lemma}(\cite[Lemma 3]{CX})\label{L1}
Let $F:M_{1}\rightarrow M_{2}$ be a smooth homotopy equivalence between two compact Riemannian manifolds, $\pi:\tilde{M}_{i}\rightarrow M_{i}$ the universal covering maps for $i=1,2$. Then, for any closed differential form $\a$ on $M_{2}$, $\pi^{\ast}(\a)$ is $d$(sublinear) on $\tilde{M}_{2}$ if the form $(F\circ\pi)^{\ast}(\a)$ is $d$(sublinear)  on $\tilde{M}_{1}$.
\end{lemma}
\begin{theorem}\label{T7}
Let $M^{}$ be a compact Riemannnian manifold of  the non-positive sectional curvature. If $M^{2n}$ is homeomorphic to a symplectic manifold  satisfies the hard Lefschetz property, then the Euler characteristic of $M$ satisfies the inequality $(-1)^{n}\chi(M^{2n})\geq0$.	
\end{theorem}
\begin{proof}
Suppose that $F: M^{2n}_{1}\rightarrow M^{2n}_{2}$ is a homotopy equivalence from a compact Riemannian manifold $M^{2n}_{1}$ of non-positive sectional curvature to a closed symplectic manifold $M^{2n}_{2}$ which satisfies  hard Lefschetz property. By an approximation argument if necessary, one may assume that $F$ is a smooth map.

Let $\pi:(\tilde{M}_{i},\tilde{g}_{i})\rightarrow(M_{i},g_{i})$ be the universal covering map, $\w$ the symplectic form on $M_{2}$. Since $(F\circ\pi)^{\ast}\w$ is a bounded closed form on $\tilde{M}^{2n}_{1}$, it follows from Proposition \ref{P1} that $(F\circ\pi)^{\ast}\w$ is $d$(sublinear). By Lemma \ref{L1}, the lifted symplectic form $\tilde{\w}$ on $\tilde{M}^{2n}_{2}$ is $d$(sublinear) as well. It follows from Theorem \ref{T1} that the $L^{2}$ cohomology of $\tilde{M}^{2n}_{2}$ is concentrated in the middle dimension. The Atiyah index theorem for covers \cite{Atiyah} then gives $(-1)^{n}\chi(M^{2n}_{2})\geq0$. Since $\chi(M^{2n}_{1})=\chi(M^{2n}_{2})$, the conclusion follows.
\end{proof}
\section*{Acknowledgements}
We would like to thank Professor H.Y. Wang for drawing our attention to the symplectic parabolic manifold and generously helpful suggestions about these.  We would also like to thank the anonymous referee for  careful reading of my manuscript and helpful comments. This work is supported by Natural Science Foundation of China No. 11801539 (Huang), No. 11701226 (Tan) and Natural Science Foundation of Jiangsu Province BK20170519 (Tan).

\end{document}